\documentclass[oneside, reqno] {amsart}
\usepackage{graphicx} % Required for inserting images
\usepackage{amsmath}
\usepackage{amssymb}
\usepackage{euscript}
\usepackage{empheq}
\usepackage{amsthm}
\usepackage{mathtools}
\usepackage[dvipsnames]{xcolor}
\usepackage{xcolor}
\usepackage{enumitem}
\usepackage{cite} %Imports biblatex package
\usepackage{mathrsfs}

\usepackage{hyperref}
\hypersetup{
    colorlinks=true,
    linkcolor=blue,
    filecolor=magenta,      
    urlcolor=cyan,
    citecolor=blue,
    pdftitle={},
    pdfpagemode=FullScreen,
    }

\DeclareMathOperator*{\esssup}{ess\,sup}
\DeclareMathOperator\supp{supp}

\newtheorem{theorem}{Theorem}[section]
\newtheorem{corollary}{Corollary}[theorem]
\newtheorem{lemma}[theorem]{Lemma}
\numberwithin{equation}{section}

\theoremstyle{definition}
\newtheorem{definition}{Definition}[section]

\title[global small solutions and smoothing]{Existence, Uniqueness, and Smoothing for Generalized EMHD}

\begin{document}
\author{Chao Wu}
\address{Department of Mathematics, Statistics and Computer Science, University of Illinois at Chicago, Chicago, IL 60607, USA}
\email{cwu206@uic.edu}

\begin{abstract}
We study the Cauchy problem for generalized electron magnetohydrodynamics (EMHD). We establish the local existence and uniqueness of solutions in critical Sobolev spaces, as well as the global existence and uniqueness for small initial data. In addition, we prove an instantaneous Gevrey smoothing effect for the corresponding solutions. Finally, we derive time decay rates for the global solutions.
\end{abstract}

\allowdisplaybreaks
\counterwithin{equation}{section}
\nopagebreak
\maketitle
%\tableofcontents

\section{Introduction}
\noindent As an important mathematical model describing fast magnetic reconnection phenomena in plasma physics, the incompressible magnetohydrodynamics (MHD) system with Hall effect is given by
\begin{equation*}\label{HMHD}
    \begin{cases}
        u_t+(u\cdot\nabla )u-(B\cdot\nabla )B+\nabla P+\nu\Lambda^2u=0\\
        B_t+(u\cdot\nabla)B-(B\cdot\nabla)u+\nabla\times((\nabla\times B)\times B)+\mu\Lambda^2B=0\\
        \nabla\cdot u=0,\quad\nabla\cdot B=0,
    \end{cases}
\end{equation*}
which is also known as Hall magnetohydrodynamics (HMHD). Having been derived in \cite{arichetogaray2011kinetic}, the HMHD differs from the classical MHD system by including the additional Hall term 
$\nabla\times((\nabla\times B)\times B)$. As customary, the unknowns of the Hall MHD are the velocity field $u$, the magnetic field $B$, and the fluid pressure $P$. The parameters $\nu$ and $\mu$ represent kinetic viscosity and magnetic resistivity, respectively. The Hall term is notoriously more challenging than other nonlinear terms in the system. To have a thorough understanding of the nonlinear structure of the Hall term, we set $u=0$ and consider the electron MHD with a more generalized version
\begin{equation}\label{EMHD}
    \begin{cases}
        B_t+\mu\Lambda^{\kappa}B+\nabla\times((\nabla\times\Lambda^{-s} B)\times B)=0\\
        \nabla\cdot B=0
    \end{cases}
\end{equation}
which is introduced in \cite{chae2023active}. The electron MHD recounts the magnetohydrodynamics with static background ion flow. It has been studied by both mathematicians and physicists. The local and global existence of Hall MHD with full dissipation has been widely studied, for example, \cite{danchin2022global}, \cite{fujii2024global}, \cite{ferreira2024well}, \cite{liu2021global}, \cite{chae2014well}, \cite{liu2022global}, \cite{houamed2021well}. The local well-posedness of Hall MHD with reduced dissipation was also studied, for example \cite{dai2022well}. For EMHD, several local well-posedness results were established, for example \cite{jeong2025wellposedness}, \cite{chae2023active}. In addition to well-posedness results, several illposedness results were also obtained, for example \cite{chae2023illposedness}, \cite{dai2024ill}, \cite{jeong2022cauchy}. Note that if $B(t,x)$ is a solution to (\ref{EMHD}) corresponding to the initial data $B_0(x)$, then
$$B_{\lambda}(t,x)=\lambda^{\kappa+s-2}B(\lambda^{\kappa}t,\lambda x) $$
is also a solution corresponding to the initial data $\lambda^{\kappa+s-2}B_0(\lambda x)$.
This implies that a critical homogeneous Sobolev space is $\dot{H}^{\sigma_c}$, where $\sigma_c:=3.5-s-\kappa.$ In the regime of $s=0$, several results were obtained. The local well-posedness of (\ref{EMHD}) with $\kappa>1$ was obtained in subcritical spaces, see \cite{chae2023active} for more details. In the regime $\kappa\le1$, ill-posedness results for (\ref{EMHD}) were established, for example, \cite{jeong2022cauchy}. We will study (\ref{EMHD}) in the supercritical and critical regime $1<\kappa<2$ with $\mu=1$. We establish the existence and
uniqueness of solutions for arbitrary initial data in $H^{\sigma_c}$, global existence when the corresponding homogeneous Sobolev norm of the initial data is sufficiently small, and establish Gevrey regularity for the unique solution (see Theorem \ref{theorem for emhd}).

\section{Mathematical preliminaries}
\subsection{Notation} A general constant which is not important to track is denoted by $C$, which may be different from line to line. We use $\epsilon$ and $\delta$ to denote small real numbers. An element in a general sequence that is not important to track is denoted by $c_j$ and the corresponding sequence is denoted by $\{c_j\}_{j\in\mathbb{Z}}$ or $\{c_j\}$, which may be different from line to line. The notation $\lesssim$ is used to replace $\leq$ up to a multiplication of a constant. We will denote by $\mathcal{B}_j$ the ball of radius $2^{j+1}$ centered at the origin and by $\mathcal{A}_j$ the ring obtained in $\{x\in\mathbb{R}^n:2^{j-1<}|x|<2^{j+1}\}$.
\subsection{Littlewood-Paley Analysis} We will apply frequency localization techniques frequently, so we recall the basics of Littlewood-Paley theory on $\mathbb{R}^n$. Define the radial function $\varphi\in C_c^{\infty}(\mathbb{R}^n)$ as
\begin{equation}
    \chi(\xi)=
    \begin{cases}
        1,\text{ for } |\xi|\leq\frac{3}{4},\\
        0,\text{ for } |\xi|\geq 1
    \end{cases}
\end{equation}
Denote $\varphi=\chi(\frac{\xi}{2})-\chi(\xi)$ and $\varphi_j(\xi)=\varphi(2^{-j}\xi)$. The $j$-th homogeneous Littlewood-Paley projection of a tempered distribution vector field $u$ on $\mathbb{R}^n$ is defined by
$$u_j(x)=\Delta_ju(x):=\mathcal{F}^{-1}(\varphi_j\hat{u})(x),$$
where $\hat{u}$ is the Fourier transform of $u$, $\mathcal{F}$ is the Fourier transformation operator, and $\mathcal{F}^{-1}$ is the inverse of $\mathcal{F}$. The unit decomposition
$$u=\sum_{j\in\mathbb{Z}}u_j$$
holds in the distribution sense. Denote
$$u_{\leq k}=\sum_{j\le k}u_j,\quad\Delta_{\tilde{k}}u=u_{\tilde{k}}:=u_{k-1}+u_k+u_{k+1}.$$
we denote
$$u_{jk}=\Delta_k\Delta_ju.$$
It is convenient to use the equivalent norm of $u$ in the space $\dot{H}^s$
$$\|u\|_{\dot{H}^s}=\left(\sum_{j\in\mathbb{Z}}2^{2sj}\|u_j\|_{L^2}^2 \right)^{\frac{1}{2}}.$$
We will use the following Bony's paraproduct formula
\begin{align*}
    \Delta_j(u\cdot v)=\sum_{|k-j|\le2}\Delta_j(u_{\le k-2}\cdot v_k)+\sum_{|k-j|\leq 2}\Delta_j(u_{k}\cdot v_{\le k-2})+\sum_{k\ge j-2}\Delta_j(\tilde{u}_k\cdot v_k),
\end{align*}
which applies to scalar functions as well. We will also frequently use the following Bernstein inequality:
\begin{lemma}
    Assume that $\supp{\hat{u}}\in\mathcal{A}_j$. Then for $1\le p<q\le\infty$, the following hold
    \begin{align*}
        &2^j\|u\|_{L^p}\simeq\|\nabla u\|_{L^p}\\
        &2^{-j}\|\nabla u\|_{L^p}\simeq\|u\|_{L^p}\\
        &\|u\|_{L^q}\le C2^{n(1/p-1/q)}\|u\|_{L^p}
    \end{align*}
\end{lemma}
\subsection{Gevrey Class}
We recall the Gevrey classes in this section. These spaces identify a scale of subspaces between the space of analytic functions and the space of smooth functions. Let $\alpha\in(0,1]$ and $\lambda>0$. Then the Gevrey operator $G_{\alpha}^{\lambda(t)}$ of order $\alpha$ and
radius $\lambda(t)$ by
\begin{align*}
    \mathcal{F}\left(G_{\alpha}^{\lambda(t)}f\right)(\xi)=e^{\lambda(t)|\xi|^{\alpha}}\hat{f}(\xi)
\end{align*}
respectively. We will also use the following notation $$e^{\lambda(t)\Lambda^{\alpha}}f=G_{\alpha}^{\lambda(t)}f.$$
We define the homogeneous Sobolev space based Gevrey norm by
$$\|f\|_{G_{\alpha,\sigma}^{\lambda}}=\|G_{\alpha}^{\lambda}f\|_{\dot{H}^{\sigma}}$$
An important observation is that if $\|f\|_{G_{\alpha,\sigma}^{\lambda}}$ is finite for some $\sigma\in\mathbb{R}$ and $\alpha,\lambda>0$, then
\begin{align*}
    \|\partial^{\beta}f\|_{\dot{H}^{\sigma}}\le\left(\frac{\beta!}{(\lambda\alpha)^{\beta}}\right)^{\frac{1}{\alpha}}\|f\|_{G_{\alpha,\sigma}^{\lambda}}
\end{align*}
for all multi-indices $\beta$.
To obtain the Gevrey smoothing results, we need the following operator.
\begin{align*}
    \mathcal{F}(E_{\alpha}^{\lambda}f)(\xi):=\left(\int_0^1e^{\lambda\tau|\xi|^{\alpha}}d\tau \right)\hat{f}(\xi)
\end{align*}
\subsection{Commutator estimates} To explore cancellations of the nonlinear term, we need the following commutators:
\begin{definition}\label{def 1}
    Let $f,g$ be vector valued functions. Then we define 
    \begin{align}\label{commutator 1}
        [\Delta_j,g\times\nabla\times]f:=\Delta_j(g\times\nabla\times f)-g\times\nabla\times f_j
    \end{align}
\end{definition}
\begin{definition}\label{def 2}
    Let $f,g$ be scalar functions. Then we define
    \begin{align}\label{commutator 2}
        [\Delta_j,g]f:=\Delta_j(gf)-g\Delta_jf
    \end{align}
\end{definition}
\noindent We will use the following lemmas frequently in a priori estimates.
\begin{lemma}\label{commutator estimate}
    Let $f$ $g$ and $h$ be vector functions. Assume that $\supp{\hat{g}}\in\mathcal{B}_j$. Then we have
    \begin{align*}
        &\langle[\Delta_jG_{\alpha}^{\lambda(t)},g\times\nabla\times]f,\nabla\times h\rangle\\
        \lesssim &2^{-j}2^{\epsilon j}\left\|\Lambda^{1+\frac{n}{2}-\epsilon}\tilde{g}\right\|_{L^{2}}\left\|\Lambda\tilde{f}\right\|_{L^{2}}\|\Lambda h\|_{L^2}\\
        &+2^{-j}2^{(\alpha-1) j}2^{\epsilon}\lambda(t)\left\|\Lambda^{1+\frac{n}{2}-\epsilon}E_{\alpha}^{\lambda}g\right\|_{L^2}\left\|\Lambda\tilde{f}\right\|_{L^{2}}\|\Lambda h\|_{L^2}
    \end{align*}
\end{lemma}
\begin{proof}
    By the definition of the commutator and Plancherel's theorem, we have
    \begin{align*}
        &\langle[\Delta_jG_{\alpha}^{\lambda(t)},g\times](\nabla\times f)\cdot\nabla\times h\rangle\\
        =&\int\bigg(\phi_j(\xi)e^{\lambda(t)|\xi|^{\alpha}}-\phi_j({\xi-\eta})e^{\lambda(t)|\xi-\eta|^{\alpha}}\bigg)\hat{g}(\eta)\times\bigg((\xi-\eta)\times\hat{f}(\xi-\eta) \bigg)\cdot\bigg(\xi\times\hat{h}(\xi)\bigg)d\eta d\xi
    \end{align*}
    Writing
    \begin{align*}
        &\phi_j(\xi)e^{\lambda(t)|\xi|^{\alpha}}-\phi_j({\xi-\eta})e^{\lambda(t)|\xi-\eta|^{\alpha}}\\
        =&\phi_j(\xi)\left(e^{\lambda(t)|\xi|^{\alpha}}-e^{\lambda(t)|\xi-\eta|^{\alpha}} \right)+\Big(\phi_j(\xi)-\phi_j({\xi-\eta}\Big)e^{\lambda(t)|\xi-\eta|^{\alpha}}
    \end{align*}
    Denote by $A(\tau)=\tau\xi+(1-\tau)(\xi-\eta)$. Then we have (See Lemma4.4 in \cite{jolly2021existence} for more details)
    \begin{align*}
        &\left|e^{\lambda(t)|\xi|^{\alpha}}-e^{\lambda(t)|\xi-\eta|^{\alpha}} \right|\\
        =&\left|\int_0^1\frac{d}{d\tau}\left(e^{\lambda|A(\tau|^{\alpha}}\right) \right|\lesssim\lambda|\eta|2^{(\alpha-1)j}e^{\lambda|\xi-\eta|^{\alpha}}\int_0^1e^{\lambda\tau^{\alpha}|\eta|^{\alpha}}d\tau
    \end{align*}
    writing
    \begin{align*}
        &\langle[\Delta_jG_{\alpha}^{\lambda(t)},g\times](\nabla\times f)\cdot\nabla\times f\rangle\\
        =&\int\frac{\phi_j(\xi)e^{\lambda(t)|\xi|^{\alpha}}-\phi_j({\xi-\eta})e^{\lambda(t)|\xi-\eta|^{\alpha}}}{|\eta|}\widehat{\Lambda g}(\eta)\times\bigg((\xi-\eta)\times\hat{f}(\xi-\eta) \bigg)\cdot\bigg(\xi\times\hat{h}(\xi)\bigg)d\eta d\xi,
    \end{align*}
    and applying H\"older's inequality, Young's inequality, and using the fact that the support of $g\in\mathcal{B}_j$, we complete the proof.
\end{proof}

\section{Statement of main theorems} 
\noindent In this section, we state the main theorems and provide immediate corollaries.
\begin{theorem}\label{theorem for emhd}
    Let $-\frac{1}{2}<s<\frac{1}{2}$ and $2-2s<\kappa<2.5-s$. Then For each $B_0\in H^{\sigma_c}$, there exits $T>0$ a unique solution $B$ of (\ref{EMHD}) such that
    \begin{align*}
        B\in C([0,T);H^{\sigma_c})\cap L^2(0,T;\dot{H}^{\sigma_c+\kappa/2}).
    \end{align*}
 Moreover, for $0<\delta<\alpha\le1$ with $\delta$ sufficiently small, there exists an increasing function $\lambda:[0,\infty)\to[0,\infty)$ with $\lambda(0)=0$ such that
    \begin{align*}
        \|B(t)\|_{G_{\alpha,\sigma_c+\delta}^{\lambda(t)}}\leq\|B_0\|_{\dot{H}^{\sigma_c}}t^{-\frac{\delta}{\kappa}}
    \end{align*}
    for all $0<t<T$, for some constant $C>0$ independent of T. Lastly, if $\|B_0\|_{\dot{H}^{\sigma_c}}$ is small enough, then $T=\infty$ is allowed.
\end{theorem}
\noindent We will see in the proof of Theorem \ref{theorem for emhd} that
\begin{align}\label{lambda}
    \lambda:=\epsilon t^{\frac{\alpha}{\kappa}}
\end{align}
where $\epsilon>0$ is chosen to be sufficiently small. With this in mind, we have the following corollary:
\begin{corollary}
    Let $-\frac{1}{2}<s<\frac{1}{2}$ and $2-2s<\kappa<2.5-s$. Then for $0<\delta<\alpha\le1$ with $\delta$ sufficiently small and $\|B_0\|_{\dot{H}^{\sigma_c}}$ sufficiently small, we have for each integer $k>0$
    \begin{align*}
        \|D^kB(t)\|_{\dot{H}^{\sigma_c+\delta}}\leq C_k\|B_0\|_{\dot{H}^{\sigma_c}}t^{-\frac{k+\delta}{\kappa}},
    \end{align*}
    for all $t>0,$ where $\lambda(t)$ is given by (\ref{lambda}).
\end{corollary}

\section{Proof of Theorem \ref{theorem for emhd}}
\noindent The proof of theorem \ref{theorem for emhd} relies on an approximating sequence determined by the following PDE
\begin{equation}\label{auxillary equation for emhd}
    \begin{cases}
        B_t+\Lambda^{\kappa}B+\nabla\times((\nabla\times\Lambda^{-s} B)\times q)=0\\
        \nabla\cdot B=0
    \end{cases}.
\end{equation}
The existence and uniqueness theorem of (\ref{auxillary equation for emhd}) is the following:
\begin{theorem}\label{auxillary theorem for emhd}
    Let $-\frac{1}{2}<s<\frac{1}{2}$ and $2-2s<\kappa<2.5-s$. Given $B_0\in H^{\sigma_c}$, suppose $q\in L^{\infty}(0,T;H^{1+\epsilon})\cap L^2(0,T;\dot{H}^{\sigma_c+\kappa/2})$. Then for a $T>0$ sufficiently small, (\ref{auxillary equation for emhd}) has a unique solution $B\in C([0,T];H^{\sigma_c})\cap L^2(0,T;\dot{H}^{\sigma_c+\frac{\kappa}{2}})$ that satisfies
    \begin{align*}
        \sup_{0\leq t\leq T}\|B(t)\|_{H^{\sigma_c}}\leq\|B_0\|_{H^{\sigma_c}}\exp{\left(C\|q\|_{L_T^2\dot{H}^{\sigma_c+\kappa/2}}\right)}.
    \end{align*}
\end{theorem}
\noindent Theorem \ref{auxillary theorem for emhd} can be proved by an artificial viscosity argument. A sketch of the proof is provided in Appendix A. With this in hand, we only need a priori estimates for (\ref{auxillary equation for emhd}). We will adapt the strategy in \cite{jolly2021existence} to obtain the relevant estimates.
\subsection{Sobolev space estimates for EMHD}We note that a direct energy estimate yields
\begin{align}\label{L2 conservation}
    \frac{1}{2}\frac{d}{dt}\|B\|_{L^2}^2+\|B\|_{\dot{H}^{\frac{\kappa}{2}}}^2\le0
\end{align}
\subsubsection{Intermediary $L^{\frac{\kappa}{\delta}}_T\dot{H}^{\sigma_c+\delta}$-estimates} To close the estimates in $L^{\infty}\dot{H}^{\sigma_c}$ and $L^2_T\dot{H}^{\sigma_c+\kappa/2}$, we need to establish intermediary estimates. Applying $\Delta_j$ to (\ref{EMHD}) and multiplying both sides of (\ref{EMHD}) by $2^{2(\sigma_c+\kappa)j}2^{sj}\Lambda^{-s}B_j$ yields
\begin{align*}
    &\frac{1}{2}\frac{d}{dt}2^{2(\sigma_c+\delta)j}\|B_j\|_{L^2}^2+2^{(2\sigma_c+2\delta+\kappa) j}\|B_j\|_{L^2}^2\\
    \leq&\sum_{|k-j|\leq2}\int2^{sj}2^{(2\sigma_c+2\delta)j}\Delta_j((\nabla\times\Lambda^{-s} B_{\leq k-2})\times q_k)\cdot\nabla\times \Lambda^{-s}B_jdx\\
    &+\sum_{|k-j|\leq2}\int2^{sj}2^{(2\sigma_c+2\delta)j}\Delta_j((\nabla\times \Lambda^{-s}B_{k})\times q_{\leq k-2})\cdot\nabla\times\Lambda^{-s} B_jdx\\
    &+\sum_{k\geq j-2}\int2^{sj}2^{(2\sigma_c+2\delta)j}\Delta_j((\nabla\times \Lambda^{-s}B_{\tilde{k}})\times q_{k})\cdot\nabla\times \Lambda^{-s}B_jdx\\
    :=&I_1+I_2+I_3
\end{align*}
Using Plancherel's theorem, H\"older's inequality and Young's inequality, we obtain
\begin{align*}
    I_1\lesssim&\sum_{|k-j|\le 2}2^{j}\|\widehat{\nabla\times\Lambda^{-s} B}_{\le k-2}\|_{L^1}2^{(\sigma_c+\delta)j}\|q_k\|_{L^2}2^{(\sigma_c+\delta-1)j}\|\nabla\times B_j\|_{L^2}\\
    \lesssim&\sum_{|k-j|\le 2}2^{(1+\epsilon)j}\|\Lambda^{-s}B\|_{\dot{H}^{2.5-\epsilon}}\|q_k\|_{\dot{H}^{\sigma_c+\delta}}\|B_j\|_{\dot{H}^{\sigma_c+\delta}}\\
    \lesssim&\sum_{|k-j|\le 2}2^{(\kappa-\delta)j}\|B\|_{\dot{H}^{\sigma_c+\delta}}\|q_k\|_{\dot{H}^{\sigma_c+\delta}}\|B_j\|_{\dot{H}^{\sigma_c+\delta}}\\
    :=&c_j2^{(\kappa-\delta)j}\|B\|_{\dot{H}^{\sigma_c+\delta}}\|q\|_{\dot{H}^{\sigma_c+\delta}}\|B_j\|_{\dot{H}^{\sigma_c+\delta}}
\end{align*}
where
\begin{align*}
    c_j=\frac{\sum_{|k-j|\le 2}\|q_k\|_{\dot{H}^{\sigma_c+\delta}}}{\|q\|_{\dot{H}^{\sigma_c+\delta}}}
\end{align*}
Using the commutator, we write
\begin{align*}
    I_2=&\sum_{|k-j|\leq2}2^{sj}2^{(2\sigma_c+2\delta)j}\langle[\Delta_j,q_{\le k-2}\times]\nabla\times \Lambda^{-s}B_k,\nabla\times \Lambda^{-s}B_j\rangle\\
    &+\sum_{|k-j|\leq2}2^{sj}2^{(2\sigma_c+2\delta)j}\langle q_{\le j-2}\times (\nabla\times \Lambda^{-s}B_{k})_j,\nabla\times \Lambda^{-s}B_j\rangle\\
    &+\sum_{|k-j|\leq2}2^{sj}2^{(2\sigma_c+2\delta)j}\langle q_{\le k-2}-q_{\le j-2}\times (\nabla\times \Lambda^{-s}B_{k})_j,\nabla\times \Lambda^{-s}B_j\rangle\\
    :=&I_{21}+I_{22}+I_{23}
\end{align*}
Note that $I_{22}=0$. To estimate $I_{21}$, we use Lemma \ref{commutator estimate}, H\"older's inequatliy, Young's inequality, and Berstein's inequality to obtain
\begin{align*}
    I_{21}
    &\lesssim\sum_{|k-j|\le2}2^{(1-s)j}\|\widehat{\Lambda q_{\le k-2}}\|_{L^1}\|B_k\|_{\dot{H}^{\sigma_c+\delta}}\|B_j\|_{\dot{H}^{\sigma_c+\delta}} \\
    &\lesssim \sum_{|k-j|\le2}2^{(\kappa-\delta)j}\|q\|_{\dot{H}^{\sigma_c+\delta}}\|B_k\|_{\dot{H}^{\sigma_c+\delta}}\|B_j\|_{\dot{H}^{\sigma_c+\delta}}\\
    &\lesssim c_j2^{(\kappa-\delta)j}\|q\|_{\dot{H}^{\sigma_c+\delta}}\|B\|_{\dot{H}^{\sigma_c+\delta}}\|B_j\|_{\dot{H}^{\sigma_c+\delta}}
\end{align*}
where
\begin{align*}
    c_j=\frac{\sum_{|k-j|\le2}\|B_k\|_{\dot{H}^{\sigma_c+\delta}}}{\|B\|_{\dot{H}^{\sigma_c+\delta}}}.
\end{align*}
We use H\"older's inequality and Bernstein's inequality to obtain
\begin{align*}
    I_{23}&\lesssim \sum_{|k-j|\le 2}2^{(\kappa-\delta)j}\|q\|_{\dot{H}^{\sigma_c+\delta}}\|B_{jk}\|_{\dot{H}^{\sigma_c+\delta}}\|B_j\|_{\dot{H}^{\sigma_c+\delta}}\\
    &=c_j\|q\|_{\dot{H}^{\sigma_c+\delta}}\|B\|_{\dot{H}^{\sigma_c+\delta}}\|B_j\|_{\dot{H}^{\sigma_c+\delta}}
\end{align*}
where
\begin{align*}
    c_j=\frac{\sum_{|k-j|\le 2}\|B_{jk}\|_{\dot{H}^{\sigma_c+\delta}}}{\|B\|_{\dot{H}^{\sigma_c+\delta}}}.
\end{align*}
Note that
\begin{align*}
    I_3&\lesssim\sum_{k\ge j-2}\int2^{(2-s)j}2^{(\sigma_c+\delta+s-1)(j-k)}\|q_k\|_{L^{\infty}}2^{(\sigma_c+\delta+s-1)k}\|\nabla \Lambda^{-s}\times B_k\|_{L^2}2^{(\sigma_c+\delta-1)j}\|\nabla\times B_j\|_{L^2}\\
    &\lesssim\sum_{k\ge j-2}\int2^{(2-s)j}2^{(\sigma_c+\delta+s-1)(j-k)}\|\Lambda^{1.5} q_{\tilde{k}}\|_{L^2}\|B_k\|_{\dot{H}^{\sigma_c+\delta}}\|B_j\|_{\dot{H}^{\sigma_c+\delta}}\\
    &\lesssim\sum_{k\ge j-2}\int2^{(2-s)j}2^{(\sigma_c+\delta+s-1)(j-k)}2^{\left(1.5-(\sigma_c+\delta)\right)k}\| q_{\tilde{k}}\|_{\dot{H}^{\sigma_c+\delta}}\|B_k\|_{\dot{H}^{\sigma_c+\delta}}\|B_j\|_{\dot{H}^{\sigma_c+\delta}}\\
    &=\sum_{k\ge j-2}\int2^{(2-s)j}2^{(\sigma_c+\delta+s-1)(j-k)}2^{\left(-2+\kappa+s-\delta)\right)k}\| q_{\tilde{k}}\|_{\dot{H}^{\sigma_c+\delta}}\|B_k\|_{\dot{H}^{\sigma_c+\delta}}\|B_j\|_{\dot{H}^{\sigma_c+\delta}}\\
    &=\sum_{k\ge j-2}2^{(2-s)(j-k)}2^{(\sigma_c+\delta+s+1)(j-k)}2^{\left(\kappa-\delta\right)(k-j)}2^{\left(\kappa-\delta\right)j}\| q_{\tilde{k}}\|_{\dot{H}^{\sigma_c+\delta}}\|B_k\|_{\dot{H}^{\sigma_c+\delta}}\|B_j\|_{\dot{H}^{\sigma_c+\delta}} \\
    &=\sum_{k\ge j-2}2^{(6.5-s-2\kappa+2\delta)(j-k)}2^{\left(\kappa-\delta\right)j}\| q_{\tilde{k}}\|_{\dot{H}^{\sigma_c+\delta}}\|B_k\|_{\dot{H}^{\sigma_c+\delta}}\|B_j\|_{\dot{H}^{\sigma_c+\delta}} \\
    &:=c_j2^{(\kappa-\delta)j}\| q\|_{\dot{H}^{\sigma_c+\delta}}\|B\|_{\dot{H}^{\sigma_c+\delta}}\|B_j\|_{\dot{H}^{\sigma_c+\delta}}
\end{align*}
where
\begin{align*}
    c_j=\frac{\sum_{k\ge j-2}2^{(6.5-s-2\kappa+2\delta){(j-k)}}\| q_{\tilde{k}}\|_{\dot{H}^{\sigma_c+\delta}}\|b_k\|_{\dot{H}^{\sigma_c+\delta}}}{\| q\|_{\dot{H}^{\sigma_c+\delta}}\|B\|_{\dot{H}^{\sigma_c+\delta}}}
\end{align*}
In conclusion, we have
\begin{align*}
    \frac{1}{2}\frac{d}{dt}\|B_j\|_{\dot{H}^{\sigma_c+\delta}}^2+2^{\kappa j}\|B_j\|_{\dot{H}^{\sigma_c+\delta}}^2
    \lesssim c_j2^{(\kappa-\delta)j}\| B\|_{\dot{H}^{\sigma_c+\delta}}\|q\|_{\dot{H}^{\sigma_c+\delta}}\|B_j\|_{\dot{H}^{\sigma_c+\delta}}
\end{align*}
Dividing both sides by $\|B_j\|_{L^2}$ and an application of Gronwall's inequality yield
\begin{align*}
    \|B_j(t)\|_{\dot{H}^{\sigma_c+\delta}}
    \lesssim& e^{-2^{\kappa j}t}\|B_j\|_{\dot{H}^{\sigma_c+\delta}}
    +c_j\int e^{-2^{\kappa j}(t-s)}2^{(\kappa-\delta)j}\|q\|_{\dot{H}^{\sigma_c+\delta}}\|B\|_{\dot{H}^{\sigma_c+\delta}}\\
    \lesssim& e^{-2^{\kappa j}t}2^{\delta j}\|B_j\|_{\dot{H}^{\sigma_c}}
   +c_j\int (t-s)^{-(1-\frac{\delta}{\kappa})}\|q\|_{\dot{H}^{\sigma_c+\delta}}\|B\|_{\dot{H}^{\sigma_c+\delta}}
\end{align*}
Taking $l^2$-norm in $j$ followed by the $L^{\frac{\kappa}{\delta}}$-norm in time, we obtain
\begin{align*}
    \|B\|_{L_T^{\frac{\kappa}{\delta}}\dot{H}^{\sigma_c+\delta}}\lesssim\mathscr{T}_1(T)+\mathscr{T}_2(T)
\end{align*}
where
\begin{align*}
    &\mathscr{T}_1(T)=\left\|\left(\sum_{j}e^{-2^{\kappa j}t}2^{2\delta j}\|B_j\|_{\dot{H}^{\sigma_c}}^2\right)^{\frac{1}{2}} \right\|_{L^{\frac{\delta}{\kappa}}}\\
    &\mathscr{T}_2(T)=\left\|\int(t-s)^{-(1-\frac{\delta}{\kappa})}\|q\|_{\dot{H}^{\sigma_c+\delta}}\|B(s)\|_{\dot{H}^{\sigma_c+\delta}}ds\right\|_{L^{\frac{\delta}{\kappa}}}
\end{align*}
By Minkowski's inequality and the Lebesgue dominated convergence theorem, we obtain
\begin{align*}
    \mathscr{T}_1\le C\|B_0\|_{\dot{H}^{\sigma_c}},\quad\lim_{T\to0}\mathscr{T}_1(T)=0
\end{align*}
By Hardy-Littlewood-Sobolev's inequality followed by H\"older's inequality, we obtain
\begin{align*}
    \mathscr{T}_2\leq& C\left\|\|q(\cdot)\|_{\dot{H}^{\sigma_c+\delta}}\|B(\cdot)\|_{\dot{H}^{\sigma_c+\delta}}\right\|_{L_T^{\frac{\kappa}{2\delta}}}\\
    \leq&C\left\|\|q(\cdot)\|_{\dot{H}^{\sigma_c+\delta}}\right\|_{L_T^{\frac{\kappa}{\delta}}}\left\|\|B(\cdot)\|_{\dot{H}^{\sigma_c+\delta}}\right\|_{L_T^{\frac{\kappa}{\delta}}}\\
    =&C\|q\|_{L_T^{\frac{\kappa}{\delta}}\dot{H}^{\sigma_c+\delta}}\|B\|_{L_T^{\frac{\kappa}{\delta}}\dot{H}^{\sigma_c+\delta}}
\end{align*}
Thus, we obtain
\begin{align*}
    \|B\|_{L_T^{\frac{\kappa}{\delta}}\dot{H}^{\sigma_c+\delta}}\lesssim\|B_0\|_{\dot{H}^{\sigma_c}}+\|q\|_{L_T^{\frac{\kappa}{\delta}}\dot{H}^{\sigma_c+\delta}}\|B\||_{L_T^{\frac{\kappa}{\delta}}\dot{H}^{\sigma_c+\delta}}
\end{align*}
\subsubsection{$L_T^2\dot{H}^{\sigma_c+\frac{\kappa}{2}}$-estimates}Recall that we have
\begin{align*}
    &\frac{1}{2}\frac{d}{dt}2^{2(\sigma_c+\delta)j}\|B_j\|_{L^2}^2+2^{(2\sigma_c+2\delta+\kappa) j}\|B_j\|_{L^2}^2\\
    \leq&\sum_{|k-j|\leq2}\int2^{sj}2^{(2\sigma_c+2\delta)j}\Delta_j((\nabla\times\Lambda^{-s} B_{\leq k-2})\times q_k)\cdot\nabla\times \Lambda^{-s}B_jdx\\
    &+\sum_{|k-j|\leq2}\int2^{sj}2^{(2\sigma_c+2\delta)j}\Delta_j((\nabla\times \Lambda^{-s}B_{k})\times q_{\leq k-2})\cdot\nabla\times \Lambda^{-s}B_jdx\\
    &+\sum_{k\geq j-2}\int2^{sj}2^{(2\sigma_c+2\delta)j}\Delta_j((\nabla\times \Lambda^{-s}B_{\tilde{k}})\times q_{k})\cdot\nabla\times \Lambda^{-s}B_jdx\\
    :=&I_1+I_2+I_3
\end{align*}
Note that
\begin{align*}
    I_3&\lesssim\sum_{k\ge j-2}2^{(2-s)j}2^{(\sigma_c+\delta+s-1)(j-k)}\|\Lambda^{1.5}q_{\tilde{k}}\|_{L^2}\|B_k\|_{\dot{H}^{\sigma_c+\delta}}\|B_j\|_{\dot{H}^{\sigma_c+\delta}}\\
    &\lesssim\sum_{k\ge j-2}2^{(2-s)j}2^{(\sigma_c+\delta+s-1)(j-k)}2^{(1.5-\sigma_c-\kappa/2)k}\|q_{\tilde{k}}\|_{\dot{H}^{\sigma_c+\kappa/2}}\|B_k\|_{\dot{H}^{\sigma_c+\delta}}\|B_j\|_{\dot{H}^{\sigma_c+\delta}}\\
    &=\sum_{k\ge j-2}2^{(2-s)(j-k)}2^{(\sigma_c+\delta+s-1)(j-k)}2^{\frac{\kappa}{2}(k-j)}2^{\frac{\kappa}{2}j}\|q_{\tilde{k}}\|_{\dot{H}^{\sigma_c+\kappa/2}}\|B_k\|_{\dot{H}^{\sigma_c+\delta}}\|B_j\|_{\dot{H}^{\sigma_c+\delta}}\\
    &=\sum_{k\ge j-2}2^{(2-s)(j-k)}2^{(\sigma_c+\delta+s-1)(j-k)}2^{\frac{\kappa}{2}(k-j)}2^{\frac{\kappa}{2}j}\|q_{\tilde{k}}\|_{\dot{H}^{\sigma_c+\kappa/2}}\|B_k\|_{\dot{H}^{\sigma_c+\delta}}\|B_j\|_{\dot{H}^{\sigma_c+\delta}}
\end{align*}
Note that by using Lemma (\ref{commutator estimate}) followed by Young's inequality and H\"older's inequality, $q_{\le k-2}$ can gain $2.5-\epsilon$ give many derivatives. Under the assumption that $\kappa>2-2s$, we have $2.5-\epsilon>\sigma_c+\kappa/2$, so there is no difficult to distribute $\sigma_c+\kappa/2$ many derivatives to the low mode $q_{\le k-2}$. Thus, by a similar estimates as before, we obtain
\begin{align*}
    &\frac{1}{2}\frac{d}{dt}\|B_j\|_{\dot{H}^{\sigma_c+\delta}}^2+2^{\kappa j}\|B_j\|_{\dot{H}^{\sigma_c+\delta}}^2\lesssim c_j2^{\frac{\kappa}{2}j}\| B\|_{\dot{H}^{\sigma_c+\delta}}\|q\|_{\dot{H}^{\sigma_c+\kappa/2}}\|B_j\|_{\dot{H}^{\sigma_c+\delta}}
\end{align*}
Dividing both sides by $2^{(\delta-\frac{\kappa}{2})j}\|B_j\|_{\dot{H}^{\sigma_c+\delta}}$ and an application of Gronwall's inequality yield
\begin{align*}
    &\|B_j(t)\|_{\dot{H}^{\sigma_c+\frac{\kappa}{2}}}
    \lesssim c_j\int e^{-2^{\kappa j}(t-s)}2^{(\kappa-\delta)j}\|q\|_{\dot{H}^{\sigma_c+\delta}}\|B\|_{\dot{H}^{\sigma_c+\frac{\kappa}{2}}}\\
\end{align*}
Taking $l^2$-norm in $j$ followed by the $L^{2}$-norm in time, we obtain
\begin{align*}
    \|B\|_{L_T^{2}\dot{H}^{\sigma_c+\delta}}\lesssim\mathscr{S}_1(T)+\mathscr{S}_2(T)
\end{align*}
where
\begin{align*}
    &\mathscr{S}_1(T)=\left\|\left(\sum_{j}e^{-2^{\kappa j+1}t}2^{\kappa j}\|B_j\|_{\dot{H}^{\sigma_c}}^2\right)^{\frac{1}{2}} \right\|_{L^{2}}\\
    &\mathscr{S}_2(T)=\left\|\int(t-s)^{-(1-\frac{\delta}{\kappa})}\|q\|_{\dot{H}^{\sigma_c+\frac{\kappa}{2}}}\|B(s)\|_{\dot{H}^{\sigma_c+\delta}}ds\right\|_{L^{2}}
\end{align*}
By Minkowski's inequality and the Lebesgue dominated convergence theorem, we obtain
\begin{align*}
    \mathscr{S}_1\le C\|B_0\|_{\dot{H}^{\sigma_c}},\quad\lim_{T\to0}\mathscr{T}_1(T)=0
\end{align*}
By Hardy-Littlewood-Sobolev's inequality followed by H\"older's inequality, we obtain
\begin{align*}
    \mathscr{S}_2\leq& C\left\|\|q(\cdot)\|_{\dot{H}^{\sigma_c+\delta}}\|B(\cdot)\|_{\dot{H}^{\sigma_c+\delta}}\right\|_{L_T^{\frac{2\kappa}{\kappa+2\delta}}}\\
    =&C\|q\|_{L_T^{2}\dot{H}^{\sigma_c+\frac{\kappa}{2}}}\|B\||_{L_T^{\frac{\kappa}{\delta}}\dot{H}^{\sigma_c+\delta}}
\end{align*}
Thus, we obtain
\begin{align*}
    \|B\|_{L_T^2\dot
    H^{\sigma_c+\frac{\kappa}{2}}}\le C\|B_0\|_{\dot{H}^{\sigma_c}}+C\|q\|_{L_T^{2}\dot{H}^{\sigma_c+\frac{\kappa}{2}}}\|B\||_{L_T^{\frac{\kappa}{\delta}}\dot{H}^{\sigma_c+\delta}}
\end{align*}
\subsubsection{$L_T^{\infty}\dot{H}^{\sigma_c}$-estimates} Recall that we have
\begin{align*}
    &\frac{1}{2}\frac{d}{dt}\|B_j\|_{\dot{H}^{\sigma_c+\delta}}^2+2^{\kappa j}\|B_j\|_{\dot{H}^{\sigma_c+\delta}}^2
    \lesssim c_j2^{\frac{\kappa}{2}j}\| B\|_{\dot{H}^{\sigma_c+\delta}}\|q\|_{\dot{H}^{\sigma_c+\kappa/2}}\|B_j\|_{\dot{H}^{\sigma_c+\delta}}
\end{align*}
Setting $\delta=0$ and an application of Young's inequality, we have
\begin{align*}
    \frac{1}{2}\frac{d}{dt}\|B_j\|_{\dot{H}^{\sigma_c}}^2+2^{\kappa j}\|B_j\|_{\dot{H}^{\sigma_c}}^2
    \lesssim&c_j2^{\frac{\kappa}{2}j}\| B\|_{\dot{H}^{\sigma_c}}\|q\|_{\dot{H}^{\sigma_c+\kappa/2}}\|B_j\|_{\dot{H}^{\sigma_c}}\\
    \lesssim&c_jC\| B\|_{\dot{H}^{\sigma_c}}^2\|q\|_{\dot{H}^{\sigma_c+\kappa/2}}^2+c2^{\kappa j}\|B_j\|_{\dot{H}^{\sigma_c}}^2
\end{align*}
Then an application of Gronwall's inequality together with (\ref{L2 conservation}) yields
\begin{align}
    \|B\|_{L_T^{\infty}H^{\sigma_c}}\leq\|B_0\|_{H^{\sigma_c}}\exp{(C\|q\|_{L_T^2H^{\sigma_c+\frac{\kappa}{2}}}^2)}
\end{align}

\subsubsection{Summary of Sobolev space estimates for EMHD} In summary, we obtain
\begin{align}\label{sobolev estimates for emhd}
    &\nonumber\|B\|_{L_T^{\frac{\kappa}{\delta}}H^{\sigma_c+\delta}}\leq C\|B_0\|_{H^{\sigma_c}}+C\|q\|_{L_T^{\frac{\kappa}{\delta}}H^{\sigma_c+\delta}}\|B\|_{L_T^{\frac{\kappa}{\delta}}H^{\sigma_c+\delta}}\\
    &\nonumber\|B\|_{L_T^2H^{\sigma_c+\frac{\kappa}{2}}}\le C\|B_0\|_{H^{\sigma_c}}+C\|q\|_{L_T^2H^{\sigma_c+\frac{\kappa}{2}}}\|B\|_{L_T^{\frac{\kappa}{\delta}}H^{\sigma_c+\delta}}\\
    &\|B\|_{L_T^{\infty}H^{\sigma_c}}\leq\|B_0\|_{H^{\sigma_c}}\exp{\left(C\|q\|_{L_T^2H^{\sigma_c+\frac{\kappa}{2}}}^2 \right)}
\end{align}
\subsection{Gevrey class estimates for EMHD} In this section, we obtain an a priori estimate in the Gevrey class. We define
\begin{align*}
    \|B\|_{X_T}:=\esssup_{0<t\leq T}t^{\frac{\delta}{\kappa}}\|B(t)\|_{G_{\alpha,\sigma_c+\delta}^{\lambda(t)}}.
\end{align*}
Note that
\begin{align}\label{calculation for gevrey estimate}
    \partial_t(G_{\alpha}^{\lambda(t)}B)=\lambda'(t)G_{\alpha}^{\lambda(t)}\Lambda^{\alpha}B+G_{\alpha}^{\lambda(t)}\partial_tB
\end{align}
Applying the operator $G_{\alpha}^{\lambda(t)}\Delta_j$ to (\ref{EMHD}), multiplying $2^{2(\sigma_c+\delta)j}\tilde{B}_j$ on both sides of (\ref{EMHD}) and using (\ref{calculation for gevrey estimate}), we obtain
\begin{align*}
    &\frac{1}{2}\frac{d}{dt}2^{2(\sigma_c+\delta)j}\|\tilde{B}_j\|_{L^2}^2+2^{2(\sigma_c+\delta)j}\|\Lambda^{\kappa/2}\tilde{B}_j\|_{L^2}^2\\
    =&\lambda'(t)2^{2(\sigma_c+\delta)j}\|\Lambda^{\alpha/2}\tilde{B}_j\|_{L^2}^2+2^{2(\sigma_c+\delta)j}2^{sj}\left\langle G_{\alpha}^{\lambda(t)}\Delta_j(q\times(\nabla\times \Lambda^{-s}B)),\nabla\times \Lambda^{-s}\tilde{B}_j\right\rangle
\end{align*}
Proceeding with the same estimates as before, together with Lemma \ref{commutator estimate}, we conclude that
\begin{align}\label{gevrey estimate}
    &\frac{1}{2}\frac{d}{dt}\|\tilde{B}_j\|_{\dot{H}^{\sigma_c+\delta}}^2+\|\Lambda^{\kappa/2}\tilde{B}_j\|_{\dot{H}^{\sigma_c+\delta}}^2\nonumber\\
   \lesssim&\lambda'(t)2^{\alpha j}\|\tilde{B}_j\|_{\dot{H}^{\sigma_c+\delta}}\|\tilde{B}_j\|_{\dot{H}^{\sigma_c+\delta}}\nonumber\\
    &+2^{(\kappa-\delta)j}\|\tilde{q}\|_{H^{\sigma_c+\delta}}\|\tilde{B}_j\|_{\dot{H}^{\sigma_c+\delta}}\|\tilde{B}_j\|_{\dot{H}^{\sigma_c+\delta}}\nonumber\\
    &+2^{(\kappa-\delta)j}\|\tilde{B}\|_{H^{\sigma_c+\delta}}\|\tilde{q}_j\|_{\dot{H}^{\sigma_c+\delta}}\|\tilde{B}_j\|_{\dot{H}^{\sigma_c+\delta}}\nonumber\\
    &+c_j\lambda(t)2^{(\kappa-\delta/2)j}\|E_{\alpha}^{\lambda}\tilde{q}_{\le j}\|_{H^{\sigma_c+\delta/2+\alpha}}\|\tilde{B}\|_{\dot{H}^{\sigma_c+\delta}}\|\tilde{B}\|_{\dot{H}^{\sigma_c+\delta}}\nonumber\\
    &+c_j\lambda(t)2^{(\kappa-\delta/2)j}\|\tilde{q}\|_{H^{\sigma_c+\delta}}\|E_{\alpha}^{\lambda}\tilde{B}_{\le j}\|_{\dot{H}^{\sigma_c+\delta/2+\alpha}}\|\tilde{B}\|_{\dot{H}^{\sigma_c+\delta}}
\end{align}
Choosing $\lambda(t):=\epsilon t^{\frac{\alpha}{\kappa}}$, dividing both sides by $2^{(\sigma_c+\delta)j}\|\tilde{B}_j\|_{L^2}$, and integrating in time, we obtain
\begin{align*}
    \|\Lambda^{\sigma_c+\delta}\tilde{B}_j(t)\|_{L^2}\leq&Ce^{-c'2^{\kappa j}t}2^{\delta j}\|\Delta_jB(0)\|_{H^{\sigma_c}}\\
    &+Cc_j\epsilon\int_0^t2^{\alpha j}e^{-c'2^{\kappa j}(t-s)}s^{\alpha/\kappa-1}\|B(s)\|_{H^{\sigma_c+\delta}}ds\\
    &+Cc_j\int_0^t2^{(\kappa-\delta)j}e^{-c'2^{\kappa j}(t-s)}\|\tilde{B}(s)\|_{H^{\sigma_c+\delta}}\|\tilde{q}(s)\|_{H^{\sigma_c+\delta}}ds\\
    &+Cc_j\int_0^t\epsilon e^{-c'2^{\kappa j}(t-s)}s^{\frac{\alpha}{\kappa}}2^{(\kappa-\delta/2)j}\|E_{\alpha}^{\lambda}\tilde{q}_{\le j}(s)\|_{H^{\sigma_c+\delta/2+\alpha}}\|\tilde{B}(s)\|_{H^{\sigma_c+\delta}}\\
    &+Cc_j\int_0^t\epsilon e^{-c'2^{\kappa j}(t-s)}s^{\frac{\alpha}{\kappa}}2^{(\kappa-\delta/2)j}\|\tilde{q}(s)\|_{H^{\sigma_c+\delta}}\|E_{\alpha}^{\lambda}\tilde{B}_{\le j}(s)\|_{H^{\sigma_c+\delta/2+\alpha}}\\
    &:=Ce^{-c'2^{\kappa j}t}2^{\delta j}\|\Delta_jB(0)\|_{H^{\sigma_c}}+c_j\epsilon K_1+c_jK_2+c_jK_3+c_jK_4+
\end{align*}
By the definition of the norm $\|\cdot\|_{X_T}$, choosing $\alpha>\delta$, we obtain
\begin{align*}
    K_1\lesssim&\int_0^t(t-s)^{-\frac{\alpha}{\kappa}}s^{\frac{\alpha}{\kappa}-1}\|\tilde{B}(s)\|_{\dot{H}^{\sigma_c+\delta}}ds\\
    =&\int_0^1t^{1-\frac{\alpha}{\kappa}}(1-s)^{-\frac{\alpha}{\kappa}}t^{\frac{\alpha}{\kappa}-1}s^{\frac{\alpha}{\kappa}-1}t^{-\frac{\delta}{\kappa}}s^{-\frac{\delta}{\kappa}}(ts)^{\frac{\delta}{\kappa}}\|B(ts)\|_{\dot{H}^{\sigma_c+\delta}}ds \\
    \lesssim&t^{-\frac{\delta}{\kappa}}\int_0^1s^{-1+\frac{\alpha-\delta}{\kappa}}(1-s)^{-\frac{\alpha}{\kappa}}ds\|B\|_{X_T}\\
    \lesssim&t^{-\frac{\delta}{\kappa}}\|B\|_{X_T}
\end{align*}
\begin{align*}
    K_2\lesssim&\int_0^t(t-s)^{\frac{\delta}{\kappa}-1}\|\tilde{q}(s)\|_{\dot{H}^{\sigma_c+\delta}}\|\tilde{B}(s)\|_{\dot{H}^{\sigma_c+\delta}}ds\\
    =&\int_0^1t^{\frac{\delta}{\kappa}}(1-s)^{\frac{\delta}{\kappa}-1}s^{-\frac{2\delta}{\kappa}}t^{-\frac{2\delta}{\kappa}}(ts)^{\frac{2\delta}{\kappa}}\|\tilde{q}(ts)\|_{\dot{H}^{\sigma_c+\delta}}\|\tilde{B}(ts)\|_{\dot{H}^{\sigma_c+\delta}}ds\\
    \lesssim&t^{-\frac{\delta}{\kappa}}\int_0^1(1-s)^{\frac{\delta}{\kappa}-1}s^{-\frac{2\delta}{\kappa}}ds\|q\|_{X_T}\|B\|_{X_T}\\
    \lesssim&t^{-\frac{\delta}{\kappa}}\|q\|_{X_T}\|B\|_{X_T}
\end{align*}
For $K_3$ and $K_4$, we Note that
\begin{align*}
    &\left\|E_{\alpha}^{\lambda(t)}\right\|_{\dot{H}^{\sigma_c+\delta/2+\alpha}}^2\\
    =&\int|\xi|^{2(\sigma_c+\delta)}\left(\int_0^1|\xi|^{\alpha-\delta/2}e^{\lambda(t)(\tau^{\alpha}-1)|\xi|^{\alpha}}d\tau \right)^2\left|\widehat{G_{\alpha}^{\lambda }f}(\xi)\right|^2d\xi\\
    \lesssim&\lambda(t)^{-2\left(1-\frac{\delta}{2\alpha}\right)}\int|\xi|^{2(\sigma_c+\delta)}\left(\int_0^1\frac{1}{(1-\tau^{\alpha})^{1-\frac{\delta}{2\alpha}}}d\tau\right)^2\left|\widehat{G_{\alpha}^{\lambda }f}(\xi)\right|^2d\xi\\
    \lesssim&t^{-2(\frac{\alpha}{\kappa}-\frac{\delta}{2\kappa})}\left\|G_{\alpha}^{\lambda}f \right\|_{H^{\sigma_c+\delta}}^2
\end{align*}
Using the above estimates for $K_3$ and $K_4$, then by the same estimates as those for $K_1$ and $K_2$, we obtain
\begin{align*}
    K_3,K_4
    \lesssim t^{-\frac{\delta}{\kappa}}\|q\|_{X_T}\|B\|_{X_T}.
\end{align*}
Summarizing the above estimates, we conclude that
\begin{align}\label{gevery inequality}
    \|\Lambda^{\sigma_c+\delta}\tilde{B}_j\|_{L^2}\leq Ce^{-c'2^{\kappa j}t}2^{\delta j}\|\Delta_j B_0\|_{\dot{H}^{\sigma_c}}+Cc_j\epsilon t^{-\frac{\delta}{\kappa}}\|B\|_{X_T}+Cc_jt^{-\frac{\delta}{\kappa}}\|q\|_{X_T}\|B\|_{X_T}.
\end{align}
Multiplying both sides of (\ref{gevery inequality}) by $t^{\frac{\delta}{\kappa}}$, taking the $l^2$-norm on $j$ and taking the supremum over $0<t\leq T$, we obtain
\begin{align*}
    \|B\|_{X_T}\leq C\mathcal{I}_T(B_0)+C\epsilon\|B\|_{X_T}+C\|q\|_{X_T}\|B\|_{X_T},
\end{align*}
where
\begin{align*}
    \mathcal{I}_T(B_0):=&\sup_{0<t\leq T}\left(\sum_{j\in\mathbb{Z}}t^{\frac{2\delta}{\kappa}}2^{2\delta j}e^{-c'2^{\kappa j+1}t}\|B_j(0)\|_{\dot{H}^{\sigma_c}}^2 \right)^{\frac{1}{2}}\\
    \leq&C\left(\sum_{j\in\mathbb{Z}}\|B_j(0)\|_{\dot{H}^{\sigma_c}}^2 \right)^{\frac{1}{2}}\leq C\|B_0\|_{\dot{H}^{\sigma_c}}.
\end{align*}
By the dominant convergence theorem, we have $$\lim_{T\to0}\mathcal{I}_T(B_0)=0.$$
Taking $\epsilon$ small enough so that
$$C\epsilon\leq\frac{1}{2},$$\label{small calI}
we obtain
\begin{align}\label{Gevery bootstrap inequaltiy}
    \|B\|_{X_T}\leq C\mathcal{I}_T(B_0)+C\|q\|_{X_T}\|B\|_{X_T}
\end{align}
\subsection{Properties of the solution to EMHD}
\noindent To show the existence, we adapt the following approximating equations:
\begin{equation}
    \begin{cases}
        \partial_tB^0+\Lambda^{\kappa}B^0=0\\
        B^0(x,0)=B_0(x)
    \end{cases}
\end{equation}
and
\begin{equation}\label{approximating system for emhd}
    \begin{cases}
        \partial_tB^{n+1}+\Lambda^{\kappa}B^{n+1}+\nabla\times((\nabla\times \Lambda^{-s}B^{n+1})\times B^n)=0\\
        B^{n+1}(0,x)=B_0(x)
    \end{cases}
\end{equation}
where $n\in\mathbb{Z}_{\{\geq0\}}$.
\subsubsection{Existence} First, we establish uniform estimates for $B^{n+1}$. Invoking a priori estimates, we conclude that there exists a bounded function $\mathscr{R}_{12}:=\mathscr{T}_1+\mathscr{T}_2$ such that
\begin{align*}
    \mathscr{R}_{12}\leq C\|B_0\|_{\dot{H}^{\sigma_c}},\quad\lim_{T\to0}\mathscr{R}_{12}(T)=0
\end{align*}
and\newline
\begin{align*}
    &\|B^0\|_{L_T^2\dot{H}^{\sigma_c+\frac{\kappa}{2}}\cap L_T^{\frac{\kappa}{\delta}}\dot{H}^{\sigma_c+\delta}}\leq\mathscr{R}_{12}(T)\\
    &\|B^{n+1}\|_{L_T^2\dot{H}^{\sigma_c+\frac{\kappa}{2}}\cap L_T^{\frac{\kappa}{\delta}}\dot{H}^{\sigma_c+\delta}}\leq\mathscr{R}_{12}(T)+C_1\|B^n\|_{L_T^2\dot{H}^{\sigma_c+\frac{\kappa}{2}}\cap L_T^{\frac{\kappa}{\delta}}\dot{H}^{\sigma_c+\delta}}\|B^{n+1}\|_{L_T^2\dot{H}^{\sigma_c+\frac{\kappa}{2}}\cap L_T^{\frac{\kappa}{\delta}}\dot{H}^{\sigma_c+\delta}}.
\end{align*}\newline
Let $\|B_0\|_{H^{\sigma_c}}$ and $T_0$ be chosen small enough such that $\mathscr{R}_{12}\leq1/(4C_1)$. Note that if $\|B_0\|_{H^{\sigma_c}}$ is small enough, $T_0$ can be chosen to be $\infty$. We obtain
\begin{align*}
    \|B^n\|_{L_T^2\dot{H}^{\sigma_c+\frac{\kappa}{2}}\cap L_T^{\frac{\kappa}{\delta}}\dot{H}^{\sigma_c+\delta}}\leq2\mathscr{R}_{12}(T),\quad n=0,1,2,\dots
\end{align*}
Applying the existence theorem of the auxiliary equations, we obtain a unique solution of (\ref{approximating system for emhd}) satisfying, for $T\in(0,T_0)$
\begin{align}
    &\|B^{n+1}\|_{L_T^2H^{\sigma_c+\frac{\kappa}{2}}\cap L_T^{\frac{\kappa}{\delta}}H^{\sigma_c+\delta}}\leq2\mathscr{R}_{12}(T),\label{converge1}\\
    &\|B^{n+1}\|_{L_T^{\infty}H^{\sigma_c}}\leq\|B_0\|_{H^{\sigma_c}}\exp{(C_2\mathscr{R}_{12}(T)^2)}\label{forlimit}
\end{align}
Next, we show that the sequence of solutions $\{B^n\}_{n\geq0}$ converges to a solution of (\ref{EMHD}). Denoting $\bar{B}^{n+1}=B^{n+1}-B^n$ and $\bar{B}_0=B_0$, we see that $\bar{B}^{n+1}$ satisfies the following equation:
\begin{equation}
    \begin{cases}
        \partial_t\bar{B}^{n+1}+\Lambda^{\kappa}\bar{B}^{n+1}+\nabla\times((\nabla\times \Lambda^{-s}\bar{B}^{n+1})\times B^n)+\nabla\times((\nabla\times \Lambda^{-s}B^n)\times\bar{B}^n)=0\\
        \bar{B}^{n+1}(0,x)=0
    \end{cases}
\end{equation}
A direct energy estimate yields
\begin{align*}
    \frac{1}{2}\frac{d}{dt}\|\Delta_j\bar{B}^{n+1}\|_{L^2}^2+\|\Lambda_j^{\frac{\kappa}{2}}\bar{B}^{n+1}\|_{L^2}^2=&2^{sj}\langle\Delta_j(B^n\times(\nabla\times\Lambda^{-s}\bar{B}^{n+1})),\Delta_j\nabla\times\Lambda^{-s}\bar{B}^{n+1}\rangle\\
    &+2^{sj}\langle\Delta_j(\bar{B}^n\times(\nabla\times\Lambda^{-s} B^n)),\Delta_j\nabla\times\Lambda^{-s}\bar{B}^{n+1}\rangle\\
    :=&J_1+J_2.
\end{align*}
By Bony's paraproduct, we obtain
\begin{align*}
    J_1=&\sum_{|k-j|\leq2}\int2^{sj}\Delta_j((\nabla\times\Lambda^{-s}\bar{B}^{n+1}_{\leq k-2})\times B_k^n)\cdot\nabla\times\Lambda^{-s}\bar{B}_j^{n+1}dx\\
    &+\sum_{|k-j|\leq2}\int2^{sj}\Delta_j((\nabla\times\Lambda^{-s}\bar{B}^{n+1}_k)\times B^n_{\leq k-2})\cdot\nabla\times\Lambda^{-s}\bar{B}_j^{n+1}dx\\
    &+\sum_{k\geq j-2}\int2^{sj}\Delta_j((\nabla\times\Lambda^{-s}\bar{B}^{n+1}_k)\times B_{\tilde{k}}^n)\cdot\nabla\times\Lambda^{-s}\bar{B}_j^{n+1}dx\\
    :=&J_{11}+J_{12}+J_{13}
\end{align*}
Using H\"older's inequality and Bernstein's inequality, we obtain
\begin{align*}
    J_{11}\lesssim2^{-\frac{\kappa}{6}j}\|B_j^n\|_{\dot{H}^{\sigma_c+\kappa/2}}\|\bar{B}^{n+1}\|_{\dot{H}^{\frac{2\kappa}{3}}}\|\bar{B}_j^{n+1}\|_{L^2}
\end{align*}
$J_{12}$ can be decomposed as follows:
\begin{align*}
    J_{12}=&\sum_{|k-j|\leq2}\int2^{sj}[\Delta_j,B_{\leq k-2}^{n}\times\nabla\times]\Lambda^{-s}\bar{B}_k^{n+1}\cdot\nabla\times\Lambda^{-s}\bar{B}_j^{n+1}dx\\
    &+\sum_{|k-j|\leq2}\int 2^{sj}B_{\leq j-2}^n\times\nabla\times(\Lambda^{-s}\bar{B}_k^{n+1})_j\cdot\nabla\times\Lambda^{-s}\bar{B}_j^{n+1}dx\\
    &+\sum_{|k-j|\leq2}\int2^{sj}(B_{\leq k-2}^n-B_{\leq j-2}^n)\times\nabla\times(\Lambda^{-s}\bar{B}_k^{n+1})_j\cdot\nabla\times\Lambda^{-s}\bar{B}_j^{n+1}\\
    :=&J_{121}+J_{122}+J_{123}
\end{align*}
Note that $J_{122}=0$. Using Lemma \ref{commutator estimate}, H\"older's inequality and Bernstein's inequality, we obtain
\begin{align*}
    J_{121},J_{123}\lesssim2^{-\frac{\kappa}{6}j}\|B^n\|_{\dot{H}^{\sigma_c+\kappa/2}}\|\bar{B}^{n+1}_j\|_{\dot{H}^{\frac{2\kappa}{3}}}\|\bar{B}_j^{n+1}\|_{L^2}
\end{align*}
Note that
\begin{align*}
    J_{13}=&\sum_{k\geq j-2}\int2^{sj}[\Delta_j,B_{\tilde{k}}^n\times\nabla\times]\Lambda^{-s}\bar{B}_{k}^{n+1}\cdot\nabla\times\Lambda^{-s}\bar{B}_j^{n+1}dx\\
    &+\sum_{j-2\leq k\leq j+1}\int2^{s j} B_{\tilde{k}}^n\times\nabla\times(\Lambda^{-s}\bar{B}_{k}^{n+1})_j\cdot\nabla\times\Lambda^{-s}\bar{B}_j^{n+1}dx\\
    :=&J_{131}+J_{132}
\end{align*}
By H\"older's inequality and Bernstein's inequality, we obtain
\begin{align*}
    J_{13}
    &\lesssim\sum_{k\ge j-2}2^{(1+\frac{\kappa}{6})(j-k)}\|B^n_{\tilde{k}}\|_{\dot{H}^{\sigma_c+\kappa/2}}\|\bar{B}_{k}^{n+1}\|_{\dot{H}^{\frac{2\kappa}{3}}}\|\bar{B}_j^{n+1\|_{L^2}} \\
    &:=c_j2^{-\frac{\kappa}{6}j}\|B^n\|_{\dot{H}^{\sigma_c+\kappa/2}}\|\bar{B}^{n+1}\|_{\dot{H}^{\frac{2\kappa}{3}}}\|\bar{B}_j^{n+1}\|_{L^2}
\end{align*}
By Bony's paraproduct, we obtain
\begin{align}
    J_2=&\sum_{|k-j|\leq2}\int2^{sj}\Delta_j((\bar{B}^n_{\leq k-2}\times(\nabla\times\Lambda^{-s} B^n_k))\cdot\nabla\Lambda^{-s}\times\bar{B}_j^{n+1}dx\\
    &+\sum_{|k-j|\leq2}\int2^{sj}\Delta_j((\bar{B}^n_{k}\times(\nabla\times\Lambda^{-s} B^n_{\leq k-2}))\cdot\nabla\times\Lambda^{-s}\bar{B}_j^{n+1}dx\\
    &+\sum_{k\geq j-2}\int2^{sj}\Delta_j((\bar{B}_{\tilde{k}}^n\times(\nabla\times\Lambda^{-s} B_k^n))\cdot\nabla\times\Lambda^{-s}\bar{B}_j^{n+1}dx\\
    =&J_{21}+J_{22}+J_{23}
\end{align}
By H\"older's inequality, Sobolev embedding and Bernstein's inequality, we obtain
\begin{align*}
    &J_{21}\lesssim c_j2^{(\frac{\kappa}{3}-\delta)j}\|\bar{B}^n\|_{\dot{H}^{\frac{2\kappa}{3}}}\|B^n\|_{\dot{H}^{\sigma_c+\delta}}\|\bar{B}_j^{n+1}\|_{L^2}\\
    &J_{22}\lesssim c_j2^{(\frac{\kappa}{3}-\delta)j}\|\bar{B}_j^n\|_{\dot{H}^{\frac{2\kappa}{3}}}\|B^n\|_{\dot{H}^{\sigma_c+\delta}}\|\bar{B}_j^{n+1}\|_{L^2}
\end{align*}
Using H\"older's inequality and Bernstain's inequality, we obtain
\begin{align*}
    J_{23}
    &\lesssim\sum_{k\ge j-2}2^{(\frac{\kappa}{3}-\delta)j}2^{(1+\delta-\frac{\kappa}{3})(j-k)}\|\bar{B}_{\tilde{k}^n}\|_{\dot{H}^{\frac{2\kappa}{3}}}\|B_k^n\|_{\dot{H}^{\sigma_c+\delta}}\|\bar{B}_j^{n+1}\|_{L^2} \\
    &:= c_j2^{(\frac{\kappa}{3}-\delta)j}\|\bar{B}^n\|_{H^{\frac{2\kappa}{3}}}\|B^n\|_{H^{\sigma_c+\delta}}\|\bar{B}_j^{n+1}\|_{L^2}
\end{align*}
Summarizing the above estimates and dividing both sides by $\|\bar{B}_j^{n+1}\|_{L^2}$ we obtain
\begin{align}\label{energy inequality for B}
    \frac{d}{dt}\|\bar{B}^{n+1}_j\|_{L^2}+c2^{\kappa j}\|\bar{B}^{n+1}\|_{L^2}\leq&c_j2^{-\frac{\kappa}{6}j}\|B^n\|_{\dot{H}^{\sigma_c+\frac{\kappa}{2}}}\|\bar{B}^{n+1}\|_{L^2}\nonumber\\
    &+c_j2^{(\frac{\kappa}{3}-\delta)j}\|B^n\|_{\dot{H}^{\sigma
    _c+\delta}}\|\bar{B}^n\|_{\dot{H}^{\frac{2\kappa}{3}}}.
\end{align}
Multiplying both sides of (\ref{energy inequality for B}) by $2^{\frac{2\kappa}{3}j}$ and integrating in time, we obtain
\begin{align*}
    \|\Lambda^{\frac{2\kappa}{3}}\bar{B}_j^{n+1}\|_{L^2}\leq&\int_0^tc_j2^{\frac{\kappa}{2}j}e^{-c2^{\kappa j}(t-s)}\|B^n(s)\|_{\dot{H}^{\sigma_c+\frac{\kappa}{2}}}\|\bar{B}^{n+1}(s)\|_{\dot{H}^{\frac{2\kappa}{3}}}ds\\
    &+\int_0^tc_j2^{(\kappa-\delta)j}e^{-c2^{\kappa j}(t-s)}\|B^n(s)\|_{\dot{H}^{\sigma_c+\delta}}\|\bar{B}^n(s)\|_{\dot{H}^{\frac{2\kappa}{3}}}ds
\end{align*}
Taking $l^2$-norm in $j$ and then $L^3$-norm in time, we obtain
\begin{align*}
    \|\bar{B}^{n+1}\|_{L^3_T\dot{H}^{\frac{2\kappa}{3}}}\leq C(I+II),
\end{align*}
where
\begin{align*}
    I(T)=&\|\int_0^t(t-s)^{-\frac{1}{2}}\|B^n(s)\|_{\dot{H}^{\sigma_c+\frac{\kappa}{2}}}\|\bar{B}^{n+1}(s)\|_{\dot{H}^{\frac{2\kappa}{3}}}ds\|_{L_T^3}\\
    II(T)=&\|\int_0^t(t-s)^{-(1-\frac{\delta}{\kappa})}\|B^n(s)\|_{\dot{H}^{\sigma_c+\delta}}\|\bar{B}^n(s)\|_{\dot{H}^{\frac{2\kappa}{3}}}ds\|_{L_T^3}
\end{align*}
Applying Hardy-Littlewood-Sobolev inequality followed by H\"older's inequality yields
\begin{align*}
    I(T)\leq&C\|B^n\|_{L_T^2\dot{H}^{\sigma_c+\frac{\kappa}{2}}}\|\bar{B}^{n+1}\|_{L_T^3\dot{H}^{\frac{2\kappa}{3}}},\\
    II(T)\leq&C\|B^n\|_{L_T^{\frac{\kappa}{\delta}}\dot{H}^{\sigma_c+\delta}}\|\bar{B}^n\|_{L_T^3\dot{H}^{\frac{2\kappa}{3}}}.
\end{align*}
So we have
\begin{align*}
    \|\bar{B}^{n+1}\|_{L_T^3\dot{H}^{\frac{2\kappa}{3}}}\leq C\|B^n\|_{L_T^2\dot{H}^{\sigma_c+\frac{\kappa}{2}}\cap L_T^{\frac{\kappa}{\delta}}\dot{H}^{\sigma_c+\delta}}\left(\|\bar{B}^{n+1}\|_{L_T^3\dot{H}^{\frac{2\kappa}{3}}}+\|\bar{B}^n\|_{L_T^3\dot{H}^{\frac{2\kappa}{3}}}\right).
\end{align*}
Let $T_0$ additionally satisfy $\mathscr{T}(T)\leq1/(6C)$. Then it implies
\begin{align}\label{converge2}
    \|\bar{B}^{n+1}\|_{L_{T_0}^3\dot{H}^{\frac{2\kappa}{3}}}\leq\frac{1}{2}\|\bar{B}^n\|_{L_T^3\dot{H}^{\frac{2\kappa}{3}}}.
\end{align}
From (\ref{converge1}), (\ref{converge2}), Sobolev interpolation, Sobolev embedding, and Kato-Ponce inequality, we conclude that there exists a function $B(t,x)$ satisfying
\begin{align*}
    &B\in L_{T_0}^{\infty}H^{\sigma_c}\cap L_{T_0}^2\dot{H}^{\sigma_c+\frac{\kappa}{2}},\\
    &B^n\overset{w^{\ast}}{\rightharpoonup}B\quad\text{in}\quad L_{T_0}^{\infty}H^{\sigma_c},\\
    &B^n\xrightarrow[]{}B\quad\text{in}\quad L_{T_0}^3\dot{H}^{\sigma},\quad \forall\sigma\in[\frac{2\kappa}{3},\sigma_c+\kappa/3].\\
    &\nabla\times((\nabla\times \Lambda^{-s}B^{n+1})\times B^n)\xrightarrow[]{}\nabla\times((\nabla\times\Lambda^{-s} B)\times B)\quad\text{in}\quad L_{T_0}^{1.5}L^2
\end{align*}

\subsubsection{Continuity in time}
Let $B$ be a solution of (\ref{EMHD}) obtained above. We have observed that
\begin{align*}
    B\in L^2(0,T^*;H^{3.5-s-\frac{\kappa}{2}})
\end{align*}
for any $T^*<T_0$. Since
\begin{align*}
    \partial_tB=-\Lambda^{\kappa}B-\nabla\times((\nabla\times\Lambda^{-s} B)\times B)
\end{align*}
By Kato-Ponce and Sobolev embedding, we can observe that
\begin{align*}
    \partial_tB\in L^1(0,T^*;L^2).
\end{align*}
So we obtain
\begin{align*}
    B\in C([0,T^*);L^2).
\end{align*}
Applying Lemma 1.4 in \cite{temam2024navier} (page 263), we obtain
\begin{align}\label{weak continuity}
    B\in C_w([0,T^*);H^{\sigma_c})
\end{align}
where we used the fact that $\kappa+s<3.5$. Using (\ref{forlimit}) and (\ref{weak continuity}), we have
\begin{align*}
    \limsup_{t\to0+}\|B(t)-B_0\|_{H^{\sigma_c}}^2=\limsup_{t\to0+}\left\{\|B(t)\|_{H^{\sigma_c}}^2+\|B_0\|_{H^{\sigma_c}}^2-2\langle B(t),B_0\rangle_{H^{\sigma_c}} \right\}\\
    \leq\limsup_{t\to0}\left\{\|B_0\|_{H^{\sigma_c}}\exp{(C_2\mathscr{R}_{12}(t)^2)}+\|B_0\|_{H^{\sigma_c}}^2-2\langle B(t),B_0\rangle_{H^{\sigma_c}} \right\}=0.
\end{align*}
This establishes the right continuity of $B$ at $t=0$. By a standard bootstrap argument, we obtain
\begin{align*}
    B\in C([0,T_0);H^{\sigma_c}).
\end{align*}
\subsubsection{Uniqueness} Let $B^{(1)}$ and $B^{(2)}$ be two solutions of (\ref{EMHD}) and denote by $\bar{B}:=B^{(1)}-B^{(2)}$. Then $\bar{B}$ satisfies the following evolution equation:
\begin{align*}
    \partial_t\bar{B}+\Lambda^{\kappa}\bar{B}+\nabla\times((\nabla\times\Lambda^{-s}\bar{B})\times B^{(1)})+\nabla\times((\nabla\times \Lambda^{-s}B^{(2)})\times\bar{B})=0
\end{align*}
Using the Littlewood-Paley analysis as before, it is not difficult to show that
\begin{align*}
    \partial_t\|\bar{B}\|_{L^2}^2+\|\Lambda
    ^{\frac{\kappa}{2}}\bar{B}\|_{L^2}^2\le&C\|\Lambda^{\sigma_c+\frac{\kappa}{2}}B^{(2)}\|_{L^2}^2\|\bar{B}\|_{L^2}^2+c\|\Lambda^{\frac{\kappa}{2}}\bar{B}\|_{L^2}^2
\end{align*}
Thus, we have
\begin{align*}
    \frac{d}{dt}\|\bar{B}\|_{L^2}^2\le C\|B^{(2)}\|_{\dot{H}^{\sigma_c+\frac{\kappa}{2}}}^2\|\bar{B}\|_{L^2}^2
\end{align*}
An application of Gr\"onwall's inequity yields $\bar{B}=0$
\subsubsection{Gevrey regularity} Invoking the a priori estimate (\ref{Gevery bootstrap inequaltiy}) in Gevrey class, we obtain
\begin{align}\label{closing gevery}
    \|B^{n+1}\|_{X_T}\leq C\mathcal{I}_T(B_0)+C\|B^n\|_{X_T}\|B^{n+1}\|_{X_T}.
\end{align}
Assume that $C\|B^n\|_{X_T}\leq\frac{1}{2}$. Then from (\ref{closing gevery}), we obtain
\begin{align*}
    \|B^{n+1}\|_{X_T}\leq 2C\mathcal{I}_T(B_0).
\end{align*}
By (\ref{small calI}), we conclude that for an arbitrary initial datum $B_0$, $T$ can be chosen sufficiently small such that for $T\leq T_0$, we have
\begin{align*}
    2C^2\mathcal{I}_T(B_0)\leq\frac{1}{2}.
\end{align*}
This condition also holds for $T=\infty$ if $\|B_0\|_{\dot{H}^{\sigma_c}}$ is small enough. Then by induction, we have the uniform-in-$n$ bound on $\|B^n\|_{X_T}$ given by
$$C\|B^n\|_{X_T}\leq\frac{1}{2}.$$
Thus, the solution $B$ obtained above also satisfies this bound.
\newline

\appendix
\section{Proof of Theorem \ref{auxillary theorem for emhd}}
\noindent We consider the following artificial viscosity regularization of (\ref{auxillary equation for emhd})
\begin{equation}\label{regularization for emhd}
        \begin{cases}
            \partial_tB+\epsilon\Lambda^4 B+\Lambda^{\kappa}B+\nabla\times((\nabla\times \Lambda^{-s}B)\times q)=0\\
            B(0,x)=B_0(x)
        \end{cases}
\end{equation}
Multiplying both sides of (\ref{regularization for emhd}) by $e^{-\epsilon\Lambda^4(t-s)}$ and integrating from $0$ to $t$ yields
\begin{align*}
    B(t)=&e^{-\epsilon\Lambda^4t}B_0-\int_0^te^{-\epsilon\Lambda^4(t-s)}\Lambda^{\kappa}B(s,x)ds\\
    &+\int_0^te^{-\epsilon\Lambda^4(t-s)}\nabla\times\Big(q(s,x)\times(\nabla\times\Lambda^{-s} B(s,x))\Big)ds
\end{align*}
This suggests that we consider the space
$$L^{\infty}_tH^{\sigma_c}\cap\{B\in L^{\infty}_tH^{\sigma_c}:B(0)=B_0\}.$$
Note that
\begin{align*}
    \int_0^t\|e^{-\epsilon\Lambda^4(t-s)}\Lambda^{\kappa}B(s)\|_{H^{\sigma_c}}ds\le C\int_0^t(t-s)^{-\frac{\kappa}{4}}\|B(s)\|_{H^{\sigma_c}}ds\le Ct^{1-\frac{\kappa}{4}}\|B\|_{L^{\infty}_tH^{\sigma_c}}
\end{align*}
and
\begin{align*}
    &\int_0^t\left\|e^{-\epsilon\Lambda^4(t-s)}\nabla\times \Big(q(s)\times(\nabla\times \Lambda^{-s}B(s))\Big)\right\|_{\dot{H}^{\sigma_c}}ds\\
    \lesssim&\int_0^t\left\||\xi|^{\sigma_c+1}e^{-\epsilon|\xi|^4(t-s)}\hat{q}*\widehat{\nabla\times \Lambda^{-s}B}(\xi)\right\|_{L^2}ds\\
    \lesssim&\int_0^t(t-s)^{-\frac{\sigma_c+1}{4}}\|B\|_{H^{1-s}}\|\hat{q}\|_{L^1}ds\\
    \lesssim&\int_0^t(t-s)^{-\frac{\sigma_c+1}{4}}\|B\|_{H^{1-s}}\|q\|_{H^{1+\epsilon}}ds\\
    \lesssim&t^{1-\frac{\sigma_c+1}{4}}\|B\|_{L_t^{\infty}H^{\sigma_c}}\|q\|_{L_t^{\infty}H^{1+\epsilon}}
\end{align*}
where we used the fact that $\sigma_c>1$. Then, for a sufficiently short time, denoted by $t^{\epsilon}$, the Banach contraction mapping theorem gives a unique solution $B^{\epsilon}\in L^{\infty}_TH^{\sigma_c}$ for (\ref{regularization for emhd}). Thanks to the uniform estimate in (\ref{sobolev estimates for emhd}), we conclude that there is a $t>0$ such that $t<t^{\epsilon}$ for all $\epsilon$. Using a similar approach, we can show that there exists some $l>0$ such that $\|\partial_tB^{\epsilon}\|_{\dot{H}^{\sigma-l}}$ is bounded uniformly in $\epsilon.$ Then, by the Aubin-Lions-Simon theorem followed by a standard diagonal argument, we get a function $B\in L_t^{\infty}H^{\sigma}$ such that for each $\varphi\in C_0^{\infty}(\mathbb{R}^3)$
$$\|\varphi(B^{\epsilon_k}-B)\|_{L_t^{\infty}H^{\sigma_c-\delta}}\to0\quad\text{as}\quad k\to\infty$$ with $\delta$ sufficiently small. Now, it is straightforward to check that $B$ is a solution to (\ref{regularization for emhd}).

\bibliography{Citations}
\bibliographystyle{plain}
\end{document}